\documentclass[]{amsart}
\usepackage{amscd,amsthm,amssymb,amsfonts,amsmath,euscript}

\usepackage{graphicx}
\theoremstyle{plain}
\newtheorem{thm}{Theorem}[section]
\newtheorem{lemma}[thm]{Lemma}
\newtheorem{prop}[thm]{Proposition}
\newtheorem{cor}[thm]{Corollary}

\theoremstyle{definition}
\newtheorem{defn}[thm]{Definition}

\theoremstyle{remark}

\newcommand{\nc}{\newcommand}

\def\makeop#1{\expandafter\def\csname#1\endcsname
  {\mathop{\rm #1}\nolimits}\ignorespaces}
\makeop{Hom}   \makeop{End}   \makeop{Aut}   \makeop{Isom}  \makeop{Pic}
\makeop{Gal}   \makeop{ord}   \makeop{Char}  \makeop{Div}   \makeop{Lie}
\makeop{PGL}   \makeop{Corr}  \makeop{PSL}   \makeop{sgn}   \makeop{Spf}
\makeop{Spec}  \makeop{Tr}    \makeop{Nr}    \makeop{Fr}    \makeop{disc}
\makeop{Proj}  \makeop{supp}  \makeop{ker}   \makeop{im}    \makeop{dom}
\makeop{coker} \makeop{Stab}  \makeop{SO}    \makeop{SL}    \makeop{SL}
\makeop{Cl}    \makeop{cond}  \makeop{Br}    \makeop{inv}   \makeop{rank}
\makeop{id}    \makeop{Fil}   \makeop{Frac}  \makeop{GL}    \makeop{SU}
\makeop{Trd}   \makeop{Sp}    \makeop{Tr}    \makeop{Trd}   \makeop{diag}
\makeop{Res}   \makeop{ind}   \makeop{depth} \makeop{Tr}    \makeop{st}
\makeop{Ad}    \makeop{Int}   \makeop{tr}    \makeop{Sym}   \makeop{can}
\makeop{length}\makeop{SO}    \makeop{torsion} \makeop{GSp} \makeop{Ker}
\makeop{Adm}
\def\makebb#1{\expandafter\def
  \csname bb#1\endcsname{{\mathbb{#1}}}\ignorespaces}
\def\makebf#1{\expandafter\def\csname bf#1\endcsname{{\bf
      #1}}\ignorespaces}
\def\makegr#1{\expandafter\def
  \csname gr#1\endcsname{{\mathfrak{#1}}}\ignorespaces}
\def\makescr#1{\expandafter\def
  \csname scr#1\endcsname{{\EuScript{#1}}}\ignorespaces}
\def\makecal#1{\expandafter\def\csname cal#1\endcsname{{\mathcal
      #1}}\ignorespaces}

\def\doLetters#1{#1A #1B #1C #1D #1E #1F #1G #1H #1I #1J #1K #1L #1M
                 #1N #1O #1P #1Q #1R #1S #1T #1U #1V #1W #1X #1Y #1Z}
\def\doletters#1{#1a #1b #1c #1d #1e #1f #1g #1h #1i #1j #1k #1l #1m
                 #1n #1o #1p #1q #1r #1s #1t #1u #1v #1w #1x #1y #1z}
\doLetters\makebb   \doLetters\makecal  \doLetters\makebf
\doLetters\makescr
\doletters\makebf   \doLetters\makegr   \doletters\makegr
     \def\qed{\qedmark\medbreak}%
\def\qedmark{{\enspace\vrule height 6pt width 5pt depth 1.5pt}}%
    
\def\Gm{{{\bbG}_{\rm m}}}   

\normalsize

\def\sep{{\rm sep}}
\makeop{Bl}

\def\Spec{{\rm Spec}\,}

\newcommand{\Z}{\mathbb Z}
\newcommand{\Q}{\mathbb Q}

\newcommand{\isoto}{\stackrel{\sim}{\longrightarrow}}
\nc{\embed}{\hookrightarrow}

\nc{\ol}{\overline}
\nc{\wt}{\widetilde}
\nc{\opp}{\mathrm{opp}}

\makeop{Ram}
\makeop{Rep}
\makeop{lcm}
\def\LD{{\rm LD}}

\begin{document}
\renewcommand{\thefootnote}{\fnsymbol{footnote}}
\setcounter{footnote}{-1}
\numberwithin{equation}{section}


\title[Embeddings over global fields]
{Embeddings of fields into simple algebras over global fields}
\author{Sheng-Chi Shih, Tse-Chung Yang and Chia-Fu Yu}
\address{
(Shih) Institute of Mathematics, Academia Sinica\\
Astronomy Mathematics Building \\
No. 1, Roosevelt Rd. Sec. 4 \\
Taipei, Taiwan, 10617}
\email{r95221018@ntu.edu.tw}

\address{
(Yang) Department of Mathematics, National Taiwan University\\
Astronomy Mathematics Building \\
No. 1, Roosevelt Rd. Sec. 4 \\
Taipei, Taiwan, 10617}
\email{d97221007@ntu.edu.tw}

\address{
(Yu) Institute of Mathematics, Academia Sinica and NCTS (Taipei Office)\\
Astronomy Mathematics Building \\
No. 1, Roosevelt Rd. Sec. 4 \\
Taipei, Taiwan, 10617}
\email{chiafu@math.sinica.edu.tw}


\date{\today}

\def\c{{\rm c}}
\def\i{{\rm i}}
\def\Mat{{\rm Mat}}

\begin{abstract}
Let $F$ be a global field, $A$ a central simple algebra over $F$ and
$K$ a finite (separable or not) field extension of $F$ with degree
$[K:F]$ dividing the degree of $A$ over $F$.
An embedding of $K$ into $A$ over
$F$ exists implies an embedding exists locally everywhere. In this
paper we give detailed discussions about
when the converse (i.e. the
local-global principle in question) may hold.

\end{abstract}

\maketitle


\section{Introduction}
\label{sec:01}

The topic on central simple algebras over global fields is a central theme
in number theory. One way to extract information of a
central simple algebra is to study its maximal subfields.
For example, this information would be useful when one attempts to
analyze the terms in the geometric side of trace formulas.
Indeed, in the geometric side of the trace formula for the multiplicative
group of a given central simple algebra,
certain major terms (elliptic ones) are described by conjugacy classes of
these maximal subfields.
This leads to a very important approach, due to Eichler
\cite{eichler:crelle55}, to calculating the class
number (of an open compact level subgroup), as the class number
can be written as the trace of the characteristic function supported on
this level subgroup.
%
In his pioneer work \cite{eichler:crelle55} Eichler established
results on the trace formula for computing the class numbers and the
type numbers of  a maximal order or a hereditary order in a given
quaternion algebra over any global field. He reduced some of the problem of
computing class numbers to computing the so called {\it
optimal embeddings} among orders in maximal subfields and those in the
quaternion algebras, which are of purely local nature.
See Eichler~\cite{eichler:crelle55} and Vign\'eras~\cite{vigneras} for
more details.


A useful tool to study maximal subfields is the Hasse principle, which
enables us to describe the properties of these subfields from local
information. A naive attempt along the same direction is figuring out
whether one can describe subfields of
a central simple algebra over a global field $F$ from local
information. In this paper we study embeddings of a finite field
extension of $F$, which does not necessarily have
the maximal degree, into a central simple algebra over $F$,
and attempted to clarify related problems in this situation.
More precisely, consider a finite-dimensional central
simple algebra $A$ over a global field $F$ and a finite (separable or
not) field
extension $K$ of $F$ whose degree divides the {\it degree} $\deg A$ of
$A$ over $K$, 
that is, $[K:F]\mid\! \sqrt{[A:F]}=\deg(A)$.
Naturally one considers the
following basic questions:
\begin{itemize}
\item [(Q1)] What is the necessary and sufficient condition for which
  the field $K$ can be $F$-linearly embedded into $A$?

\item [(Q2)] What is the necessary and sufficient condition for which
  the $F_v$-algebra $K_v$ can be $F_v$-linearly embedded into the
  $F_v$-algebra $A_v$? Here $v$ is a place of $F$,
  $F_v$ denotes the completion
  of $F$ at $v$, $K_v:=K \otimes_F F_v$ and $A_v:=A \otimes_F F_v$.

\item [(Q3)] Does the local-global principle for embedding
the field $K$ into $A$ as $F$-algebras hold? That is, if $K_v$ can be
$F_v$-linearly embedded in $A_v$ for all places $v$ of $F$,
then can $K$ be $F$-linearly embedded in $A$?
\end{itemize}

We say
that the Hasse principle for a triple $(K,A,F)$
as above holds if the
question (Q3) for $K$ and $A$ over a global field $F$ has a
positive answer.
When the field extension $K$ has the maximal degree
$\deg A$, this result is well-known
and useful; see Pierce~\cite[Section 18.4]{pierce},
Prasad-Rapinchuk \cite[Proposition
A.1]{prasad-rapinchuk:metakernel96} and
\cite[Proposition~2.7]{prasad-rapinchuk:embed10}.
However, the case when $K$ does not have the maximal degree is not
explored in the literature.



We now outline the contents of this paper.
In the first part of this paper
we answer the questions (Q1), (Q2) and (Q3).
In \cite{yu:embed}, the third named author of the present paper
studies the problem of
embeddings of one semi-simple algebra into another one over an
arbitrary ground field. In Section 2 we recall some
results of embeddings obtained in
\cite{yu:embed} and provide some proofs of them for
the reader's convenience.
These give rise
to a numerical criterion for the question (Q1) over an arbitrary field; see
Lemma~\ref{24} for the precise statement.
In Section 3 we apply results of Section~\ref{sec:02}
to the case where the base
field is either a local field or a global field.
This yields a more explicit numerical criterion for (Q1) and (Q2),
respectively. Using these
criteria, we then show that for any global field $F$,
there is a Galois extension $K$ over $F$
of degree $8$ and a central simple algebra $A$ over $F$ of degree 24
for which the Hasse principle does not hold;
see an example in Section~\ref{sec:34}. Also see
Proposition~\ref{14} below for other cases for which the Hasse
principle fails .

We now describe the results in this part.
Let $K$ and $A$ be as above. Let $A=\Mat_n(\Delta)$, where $\Delta$ is
the division part of $A$ and $n:=\c(A)$ is called the {\it capacity} of
$A$ (see Definition~\ref{21}). Recall the degree $\deg (\Delta)$ of $\Delta$ is $\sqrt{[\Delta:F]}$.
For any place $v$ of $F$, let $K_v:=K\otimes_F F_v=\prod_{w|v} K_w$
be the product of local fields. Let $V^F$ and $V^K$
denote the sets of places of
$F$ and $K$, respectively.
Let $d_v$ be the order of the class $[A_v]$ in the
Brauer group of $F_v$, which is also the index of $A_v$.
By an $F$-embedding of $K$ into $A$ we mean an
$F$-linear embedding of $K$ into $A$. 
The notion of $F_v$-embedding of $K_v$ into $A_v$ is defined similarly.


\begin{thm}\label{11}
Let the notation be as above.
\begin{enumerate}
\item The set of $A_v^\times$-conjugacy classes of
 $F_v$-embeddings of $K_v$ into $A_v$ is
 in bijection with the following set
 \[ \calE_{F_v}(K_v, A_v):=\left \{(x_w)_{w|v}\mid x_w\in \bbN,\
   \sum_{w|v}
 \ell_w x_w=\c(A_v)\right \}, \]
 where \[ \ell_{w}:=[K_w:F_v]/\gcd([K_w:F_v],\deg(\Delta)). \]
 In
 particular, there is an $F_v$-embedding of $K_v$ into $A_v$ if
  and only if the finite set $\calE_{F_v}(K_v, A_v)$ is non-empty.
\item There is an $F$-embedding of $K$ into $A$ if and only if
  \[ [K:F]\mid n\,\c(\Delta\otimes_F K), \] where $\c(\Delta\otimes_F
  K)$ is the
  capacity of the central simple algebra $\Delta\otimes_F K$ over
  $K$.
\item The number  $\c(\Delta\otimes_F K)$ can be computed as follows:
  \begin{enumerate}
  \item[(i)] The degree $\delta_0$ of $\Delta$
    is the least common
    multiple of all $d_v's$ for $v\in V^F$ (the Hasse-Brauer-Noether
    theorem).
  \item[(ii)] For each $w|v$, the order $d_w'$ of the class
  $[A_v\otimes_{F_v} K_w]$
    in the Brauer group of $K_w$ is given by
\[ d_w'=d_v/\gcd([K_w:F_v], d_v). \]
  \item[(iii)] The order $\delta'$ of the class $[A\otimes_F K]$ in
    the Brauer group of $K$ is the least common multiple of $d_w'$ for
    all $w\in V^K$ (the Hasse-Brauer-Noether theorem).
  \item[(iv)] The capacity  $\c(\Delta\otimes_F K)$ is given by  
  $\c(\Delta\otimes_F K)=\delta_0/\delta'$.
  \end{enumerate}
\end{enumerate}

\end{thm}

The proof of Theorem~\ref{11} is given in Sections \ref{sec:31} and
\ref{sec:32};  see especially Propositions~\ref{32}, \ref{33} and
\ref{34}.

In the last part of Section~\ref{sec:03}, we give a necessary and
sufficient condition for a pair $(K,A)$
so that the Hasse principle in question holds. Here we fix a global
field $F$. We associate to each pair $(K,A)$ an element
\[ \bar \bfx=(\bar \bfx_w)_{w\in V^K} \in \bigoplus_{w\in V^K} \Q/\Z \]
 as follows.
For any $w\in V^K$, put
\[ \bfx_w:=\frac{\c(A_v)\cdot \gcd([K_w:F_v], d_v)}{[K:F]}\in \Q_{>0} \]
and let $\bar \bfx_w$ be the class of $\bfx_w$ in $\Q/\Z$, where
$v$ is the place of $F$ below $w$.


We show that the element $\bar {\bfx}$ is an
obstruction class to the Hasse principle for $(K,A)$, i.e. if this
class does not vanish, then the Hasse principle for $(K,A)$ fails. We
also show that the vanishing of $\bar {\bfx}$ is the only obstruction
to the Hasse principle. Namely, we have the following result
(Theorem~\ref{36}).

\begin{thm} \label{12} Notations as above.
  An $F$-embedding of $K$ into $A$ exists if and only if an
  $F_v$-embedding of
  $K_v$ into $A_v$ exists for all $v\in V^F$ and that
  the element $\bar \bfx$ vanishes.
\end{thm}

The second part of this paper deals with the Hasse principle for a
family of pairs $(K,A)$ over a fixed global field $F$
with both degrees $[K:F]$ and $\deg(A)$
constant. Let $(k,\delta)$ be a pair of positive integers with $k \mid
\delta$. We say that the Hasse principle for the pair $(k,\delta)$
holds if for any finite field extension $K$ over $F$ of degree $k$ and
any central simple algebra $A$ over $F$ of degree $\delta$, the Hasse
principle for the pair $(K,A)$ holds.

We describe the results in this part.
Let $(k,\delta)$ be a pair of positive integers as above.
For each partition $\lambda$ of $k$
\[ \lambda=(k_1,\dots, k_t), \quad k_1\le \dots \le k_t,
\quad \sum_{i=1}^t k_i=k \]
and a positive divisor $s|\delta$,
we define a finite set
\[ \calE(\lambda,s):=\{(x_i)\in \bbN^t \mid \sum_{i=1}^t \ell_i
x_i=s\}, \]
where $\ell_i:=k_i/\gcd(k_i,d)$, and a vector $\bfx(\lambda,s)$
in $\Q^t$, where
\[ \bfx(\lambda, s):=(\bfx_i), \quad
\bfx_i=\bfx_i(\lambda,s):=s\gcd(k_i,d)/k, \quad \forall\, i=1,\dots, t. \]

Let $\LD (k,\delta)$ be the set consisting of all pairs $(\lambda,
 s)$, where
 \begin{itemize}
 \item $\lambda$ is a partition of $k$, and
 \item $s$ is a positive integer dividing $\delta$,
 \end{itemize}
such that the set $\calE(\lambda,s)$ is non-empty.
(``\LD'' stands for all possible local decompositions).
The following result
gives a simple way to check the Hasse principle for a pair $(k,\delta)$.

\begin{thm}\label{13}
Let $(k,\delta)$ be a pair of positive integers with $k\mid\delta$.
\begin{itemize}
\item [(1)] Suppose for all elements $(\lambda,s)\in \LD(k,\delta)$,
  the vector $\bfx(\lambda,s)$ has integral components. Then the
  Hasse principle for $(k,\delta)$ holds.
\item [(2)] Suppose there is an element $(\lambda,s) \in
  \LD(k,\delta)$ such that $\bfx_i(\lambda,s)\not\in \bbN$ for some
  $i$. Then there is a finite separable field extension $K$ of degree
  $k$ over $F$ and a central division algebra $A$ of degree
  $\delta$ over $F$ so that the
  Hasse principle for $(K,A)$ fails. In particular, the Hasse
  principle for the pair $(k,\delta)$ fails.
\end{itemize}
\end{thm}

The proof of Theorem~\ref{13} uses Theorem~\ref{12}, the Hilbert
irreducibility theorem and the global class field theory. Based on
Theorem~\ref{13}, we prove the following result (Proposition~\ref{46}
and Corollary~\ref{47}), which also shows
that (Q3) has a negative answer in general.

\begin{prop}\label{14}\
\begin{enumerate}
\item Let $\delta=p_1^{n_1}\cdots p_r^{n_r}$,
and let $k$ be a positive integer with $k>1$ and $k \mid \delta$, where
$p_i$\!\!'s are distinct prime divisors of $\delta$.
Assume that $k \leq \delta/p_i^{n_i}$
and $p_i \mid k$ for some $1\leq i\leq r$ (so $r\ge 2$).
Then the Hasse principle
for the pair $(k,\delta)$ does not hold.
\item    Let $\delta$ be a positive integer divisible by
at least two primes. Then
  there is a positive integer $k$ with $k\mid \delta$ such that the
  Hasse principle for the pair $(k,\delta)$ does not hold.
\end{enumerate}
\end{prop}

Clearly Proposition~\ref{14} (2) follows from Proposition~\ref{14} (1).
Proposition~\ref{14} tells us that when $\delta$ is divisible by two
primes and $k$ is ``small'' comparable to $\delta$,
the Hasse principle for $(\delta,k)$ fails.

In the last section we discuss the Hasse principle for a geometric
orbit of the variety of embeddings of $K$ into $A$ that is defined
over $F$, where $K$ is a {\it finite etale commutative $F$-algebra}. 
Let $X$ be the $F$-scheme of embeddings of $K$ into $A$ (see
Section~\ref{sec:07}). The group $H=\GL_1(A)$ acts on $X$ by
conjugation. We have the following result.

\begin{thm}[The Hasse principle]\label{110}
Let $\ol X_0$ be a geometric orbit of $H\otimes_F F^\sep$ that is
defined over $F$ and let $X_0$ be the $F$-subscheme of $X$ whose base
extension to $F^\sep$ is $\ol X_0$.  If
$X_0(F_v)\neq \emptyset$ for all $v\in V^F$, then $X_0(F)\neq \emptyset$.
\end{thm}

The proof relies on results of Springer \cite{springer:h2}, Douai  and
Borovoi \cite{borovoi:duke1993} about the second
non-abelian Galois cohomology and neutral classes. 
Among early results in this area one can mention the
validity of the Hasse principle for projective homogeneous spaces
due to Harder \cite{harder:hp68} and for symmetric homogeneous spaces of
absolutely simple simply-connected groups due to Rapinchuk
\cite{rapinchuk:hp87}. Later, Borovoi in a series of papers developed
cohomological methods for analyzing the Hasse principle for
homogeneous spaces with connected stabilizers, of an arbitrary
connected group whose maximal semi-simple subgroup is
simply-connected; see Borovoi \cite{borovoi:crelle1996,
  borovoi:mathann1999}. In \cite[Appendix A]{prasad-rapinchuk:embed10}
Prasad and Rapinchuk used the cohomological method to study the
local-global principle for embeddings of maximal subfields into central
simple algebras over global fields {\it with involutions}. \\


This paper is organized as follows. In Section~\ref{sec:02} we collect
and show some general embedding results over any field
based on \cite{yu:embed}. In Section~\ref{sec:03},
we give a more detailed study about embeddings of a
field extension $K$ into a central simple algebra $A$ over a global
field.
We determine for which pair $(K,A)$ the Hasse principle holds. In
particular, we answer the questions (Q1), (Q2) and (Q3).
In Section~\ref{sec:04}, 
we work on the Hasse principle for a given
pair $(k,\delta)$ of positive integers with $k$ dividing $\delta$.
In Section~\ref{sec:07} we give the proof of Theorem~\ref{110} 
following a referee's suggestion.




\section{General embedding results}
\label{sec:02}

\subsection{Setting}
\label{sec:21}

Let $F$ denote the ground field, which is arbitrary in this section.
All $F$-algebras considered in this paper are assumed to be
finite-dimensional as $F$-vector spaces and have the identity.
As the standard convention,
an $F$-algebra homomorphism is a ring homomorphism over $F$; in
particular, it sends the identity of the source to the identity of
the target.

We recall the following definition for central simple algebras; see
\cite{reiner:mo}.

\begin{defn}\label{21}
The {\it degree}, {\it capacity}, and {\it index} of a
central simple algebra $A$ over $F$ are defined as
\[ \deg(A):=\sqrt{[A:F]},\quad \c(A):=n,\quad \i(A):=\sqrt{[\Delta:F]}, \]
if $A\cong \Mat_n(\Delta)$, where $\Delta$ is a division algebra over
$F$, which is uniquely determined by $A$ up to isomorphism.
The algebra $\Delta$ is also called the {\it division part} of $A$.
\end{defn}

For the convenience of discussion, we introduce the following
definition and notations. 

\begin{defn}\label{22} \

\begin{enumerate}
  \item Let $V$ be a finite-dimensional vector space
  over $F$, and
  $A$ a finite-dimensional arbitrary $F$-algebra. We say that {\it
  $V$ admits an $A$-module structure}
  if there is a right (or left) $A$-module structure on $V$.
  If $B$ is any $F$-subalgebra of $A$ and $V$ is already a right
  (resp. left)
  $B$-module, then by saying $V$ admits an $A$-module structure
  we mean that the right
  (resp. left) $A$-module structure on $V$ is required to be
   compatible with the underlying $B$-module structure on $V$.

 \item  For any two
$F$-algebras $A_1, A_2$, let $\Hom_F(A_1,A_2)$ denote the set of
$F$-algebra homomorphisms from $A_1$ to $A_2$, and let
\[ \Hom_F^*(A_1,A_2)\subset \Hom_F(A_1, A_2) \]
be the subset consisting of
embeddings of $A_1$ into $A_2$.
For two maps $\varphi_1, \varphi_2\in \Hom_F(A_1,A_2)$,
we say $\varphi_1$ and $\varphi_2$ are {\it
  equivalent} if there is an element $b\in A_2^\times$ such
  that $\varphi_2=\Int(b)\circ \varphi_1$. That is,
  $\varphi_2(a)=b\, \varphi_1(a) \,b^{-1}$ for all $a\in A_1$. Then
$A_2^\times \backslash \Hom_F(A_1, A_2)$ is the set of equivalence
classes of $F$-algebra homomorphisms from $A_1$ to $A_2$. Write
\[ \calO_{A_1,A_2}:=A_2^\times \backslash \Hom_F(A_1, A_2)\quad
\text{and} \quad \calO^*_{A_1,A_2}:= A_2^\times \backslash \Hom^*_F(A_1,
A_2) \]
for the orbits spaces.

We often simply write $A_1\otimes A_2$ for $A_1\otimes_F A_2$ if the
ground field $F$ is understood.

\end{enumerate}

\end{defn}

\subsection{General embedding lemmas}
\label{sec:22}
Let $A$ be a central simple algebra over $F$. We realize
$A$ as $\End_\Delta(V)$, the endomorphism algebra of $V$,
where $\Delta$ is the division part of $A$ and $V$
is a right $\Delta$-module of finite rank.
Let $A'$ be another simple $F$-algebra with center $K$.
Then there is an $F$-algebra
homomorphism $\varphi:A'\to A$ if and only if there is an
$(A',\Delta)$-bimodule structure on $V$, or $V$ admits a right
$\Delta\otimes_F {A'}^{o}$-module structure.
Here ${A'}^{o}$ denotes the opposite algebra of $A'$.
Suppose $n:=\dim_\Delta V$ and
\[ \Delta\otimes_F {A'}^{o}\simeq (\Delta\otimes_F K)\otimes_K
{A'}^{o}\simeq \Mat_c(\Delta'),\]
where $\Delta'$ is the division part of the central simple algebra
$\Delta\otimes_F {A'}^{o}$ over $K$.
We have
\begin{equation}
  \label{eq:21}
  [\Delta:F][A':F]=c^2 [\Delta':F] \quad\text{and}\quad \dim_F V=n
  [\Delta:F].
\end{equation}
Then $V$ admits a right $\Delta\otimes_F {A'}^{o}$-module structure
if and only if
\begin{equation}
  \label{eq:22}
  c [\Delta':F]\, |\, \dim_F V.
\end{equation}
This is equivalent to $[\Delta:F][A':F]\,|\, c n [\Delta:F]$, or
equivalently
\begin{equation}
  \label{eq:23}
  [A':F]\,|\, n c.
\end{equation}
We have proved the following result.
\begin{lemma}\label{23}
  Let $A$ be a central simple $F$-algebra and $A'$ a simple
  $F$-algebra. Suppose $A\simeq \Mat_n(\Delta)$.
  Then there exists an $F$-embedding of $A'$ into $A$ if and only if
  \begin{equation}
    \label{eq:24}
    [A':F]\,|\, n\cdot \c(\Delta\otimes_F {A'}^{o}).
  \end{equation}
\end{lemma}

Next we consider the case where $A'$ is a {\it semi-simple}
$F$-algebra instead.
Write $A'=\prod_{i=1}^s A'_i$ into simple factors and let $K_i$ be the
center of $A'_i$.
The existence of an
$F$-embedding of $A'$ into $A$ is equivalent to that
there is an $(A',
\Delta)$-bimodule structure on $V$. This means that there is a
decomposition of $V$ into  $\Delta$-submodules
\begin{equation}
  \label{eq:25}
  V=V_1\oplus \dots \oplus V_s
\end{equation}
so that each $V_i$ is a {\it non-zero}  $(A_i', \Delta)$-bimodule. Put
$n_i:=\dim_\Delta V_i$ and let $c_i$ be the capacity of the central
simple algebra
\[ \Delta\otimes_F {A'_i}^{o}= (\Delta\otimes_F K_i) \otimes_{K_i}
{A'_i}^{o} \]
over $K_i$. Then we get the conditions
\begin{equation}
  \label{eq:26}
  n=\sum_{i=1}^s n_i \quad \text{and} \quad [A_i':F]\,|\,n_i c_i, \quad
  \forall\, i=1,\dots, s.
\end{equation}

This yields the following criterion for embeddings.

\begin{lemma}\label{24}
  Let $A$ and $A'=\prod_{i=1}^s A'_i$ be as above. Then there is an
  embedding of the $F$-algebra $A'$ into $A$ if and only if there
  are positive integers $n_i$ for $i=1,\dots, s$ such that
\begin{equation*}
  n=\sum_{i=1}^s n_i \quad \text{and}
  \quad [A_i':F]\,|\, n_i c_i, \quad
  \forall\, i=1,\dots, s,
\end{equation*}
where $c_i$ the capacity of the central
simple algebra $\Delta\otimes_F {A'_i}^{o}$ over $K_i$.
\end{lemma}

\begin{lemma}\label{25}
  Let $A$ and $A'=\prod_{i=1}^s A'_i$ be as above. Let $\varphi,
  \varphi'$ be two maps in $\Hom_{F}(A',A)$, and let
  $V_\varphi$ and $V_{\varphi'}$ be the associated
  $(A',\Delta)$-bimodules
  underlying the space $V$. Then $\varphi$ and $\varphi'$ are
  equivalent if and only if the $(A',\Delta)$-bimodules $V_\varphi$
  and $V_{\varphi'}$ are isomorphic.
\end{lemma}

\begin{proof}
  See \cite[Lemma 3.2]{yu:embed}
\end{proof}

\subsection{Maximal degree field extension case}
\label{sec:23}
We apply Lemmas~\ref{23} and \ref{24} to the case where the
semi-simple algebra $A'$ is commutative and obtain the following
well-known result.
This is also a consequence of a result of Chuard-Koulmann and Morales
\cite[Proposition 4.3]{chuard-koulmann-morales}.


\begin{lemma}\label{26} \
\begin{enumerate}
\item Let $A$ be a central simple algebra over $F$ and $K$ is a field
  extension of $F$ with $[K:F]=\deg(A)$. Then there exists an
  $F$-embedding of $K$ into $A$ if and only if $K$ splits $A$.
\item Let $A$ be a central simple algebra over $F$ and
  $K=\prod_{i=1}^s  K_i$ is commutative semi-simple
  $F$-algebra with $[K:F]=\deg(A)$. Then there exists an
  $F$-embedding of $K$ into $A$ if and only if each $K_i$ splits $A$.
\end{enumerate}
\end{lemma}

\begin{proof}
  (1) By Lemma~\ref{23}, the set $\Hom_F(K,A)$ is non-empty if and
  only if $[K:F]\,|\, nc$, where $c:=\c(\Delta\otimes K)$. If $K$
  splits $\Delta$, then $c=\deg(\Delta)$ and hence
  $[K:F]=n\deg(\Delta)=nc$. Therefore, $\Hom_F(K,A)$ is non-empty.
  Suppose $[K:F]|nc$. Then $\deg(\Delta)|c$
  and $\deg(\Delta)=c$. This shows that $K$ splits $\Delta$.

 (2) Suppose there is an $F$-embedding of $K$ into
  $A$. Then there are positive integers $n_i$ with $n=\sum n_i$ and
  there is an embedding of $K_i$ into $\Mat_{n_i}(\Delta)$. Since
  $[K_i:~F]\,|\, n_i \deg(\Delta)$, it follows from
   \[ [K:F]=\deg (A) = \sum_i n_i \deg(\Delta) \ge \sum_i
  [K_i:F]=[K:F] \]
that $[K_i:F]=n_i \deg(\Delta)$ for each $i=1,\dots,s$. Therefore,
$K_i$ splits $\Delta$. Conversely, if
$K_i$ splits $\Delta$ for each $i$, then $[K_i:F]=m_i \deg (\Delta)$
for a positive integer $m_i$ and $c_i:=\c(\Delta\otimes_F
K_i)=\deg(\Delta)$. Then we have
\[ n=\sum_i m_i, \quad \text{and}\quad [K_i:F]\, |\, m_i c_i, \quad
\forall\, i=1,\dots, s. \]
It follows from Lemma~\ref{24} that there is an embedding of $K$
into $A$. \qed
\end{proof}

\section{Answers to (Q1) and (Q2) by numerical invariants}
\label{sec:03}

In this section we study $F$-embeddings of $K$ into $A$ over $F$,
 where $F$ is either a local or global field,
 $K$ is a commutative semi-simple algebra over $F$ and $A$ is a
 central simple algebra.




\subsection{Local results}
\label{sec:31}

Let $F$ denote a local field.

\begin{lemma}\label{31}
  Let $A=\End_\Delta(V)=\Mat_n(\Delta)$ be a central simple
  algebra over $F$.
\begin{enumerate}
\item Let $K$ be a finite field extension of $F$. The following
  statement are equivalent:
  \begin{itemize}
  \item [(a)] There exists an
  embedding of $K$ into $A$ over $F$.
  \item [(b)] $[K:F]\,|\, n\cdot \c(\Delta\otimes_F K)$.
  \item [(c)] $[K:F]\,|\, n \deg(\Delta)$.
  \end{itemize}
\item Let $K=\prod_{i=1}^s K_i$ be a commutative semi-simple algebra
  over $F$. Then there exists an embedding of   $K$ into $A$ over $F$
  if and
  only if there are positive integers $n_i$ for $i=1,\dots, s$ such
  that
  \begin{equation}
    \label{eq:31}
    n=\sum_{i=1}^s n_i \quad \text{and} \quad [K_i:F]\,|\, n_i \deg
(\Delta), \quad \forall\, i=1,\dots, s.
  \end{equation}
\end{enumerate}
\end{lemma}

It follows from Lemma~\ref{24} that  the statements (a) and (b) are
equivalent. The implication $(b)\implies (c)$ is trivial. Put
$\delta:=\deg(\Delta)$ and $k:=[K:F]$. If
$\inv(\Delta)=a/\delta$ with $\gcd(a,\delta)=1$, then (see \cite{serre:lf})
\[ \inv(\Delta\otimes_F
K)=[K:F]\inv(\Delta)=\frac{ak}{\delta}=\frac{a'}{\delta'}, \quad
\text{with \ } \gcd(a',\delta')=1, \]
where $\delta=\delta' c$, $ak=a'c$, and $c:=\gcd(k,\delta)$.
It follows that
\begin{equation}
  \label{eq:32}
  \c(\Delta\otimes_F K)=\gcd([K:F],\deg(\Delta)).
\end{equation}
Note that $\gcd(\delta',k)=1$, so we have
\[ k\, |\ n \delta\iff k\,|\, nc \delta' \iff k\,|\, nc. \]

The statement Lemma~\ref{31} (2) follows from Lemma~\ref{24} and
Lemma~\ref{31} (1). \\

Now consider the case where $K=\prod^s_{i=1} K_i$ is a commutative
semi-simple $F$-algebra. Put
\begin{equation}
  \label{eq:33}
  c_i:=\gcd([K_i:F], \deg (\Delta)) \quad \text{and}\quad  \ell_i:=[K_i:F]/c_i.
\end{equation}
For any positive integer $n_i$, we have
\begin{equation}\label{eq:34}
  [K_i:F]\,|\, n_i \deg(\Delta) \iff \ell_i\,|\, n_i.
\end{equation}
Put
\begin{equation}
  \label{eq:35}
  \calE_F(K,A):=\{ x=(x_1,\dots, x_s)\in \bbN^s\, |\, \sum_{i=1}^s \ell_i
  x_i=\dim_\Delta V \, \}.
\end{equation}
If a tuple $\bfn=(n_1,\dots, n_s)$ is a solution to (\ref{eq:31}), then
the tuple $x=(x_1,\dots, x_s)$, where $x_i:=n_i/\ell_i$, is an element
in $\calE_F(K,A)$. Conversely, any element $x$ in $\calE_F(K,A)$ gives
a solution $\bfn$ to (\ref{eq:31}) by setting $n_i=\ell_i x_i$.
Recall that $A^\times \backslash \Hom_F^*(K,A)$ is the set of
equivalence classes of embeddings of $F$-algebras from $K$ into $A$.

\begin{prop}\label{32}
  There is a natural bijection
  \begin{equation}
    \label{eq:36}
    e: A^\times \backslash \Hom_F^*(K,A)\isoto \calE_F(K,A).
  \end{equation}
\end{prop}
\begin{proof}
  Let $\varphi$ and $\varphi'$ be two maps in $\Hom^*_F(K,A)$, and
  let $V_\varphi$ and $V_{\varphi'}$ be the induced
  $(K,\Delta)$-bimodule structures on $V$. Write
\[ V_\varphi=V_1\oplus \dots \oplus V_s \quad \text{and}\quad
  V_{\varphi'}=V'_1\oplus \dots \oplus V'_s, \]
  where $V_i$ and $V'_i$ are $(K_i,\Delta)$-bimodules.
  We have shown (Lemma~\ref{25}) that $\varphi$ and $\varphi'$ are
  equivalent if and
  only if $V_\varphi$ and  $V_{\varphi'}$ are isomorphic as
  $(K,\Delta)$-bimodules, equivalently, $V_i\simeq V'_i$ as $(K_i,
  \Delta)$-bimodules for $i=1,\dots, s$. Since each $\Delta\otimes_F K_i$
  is simple, the latter is the same as the condition $\dim_\Delta
  V_i=\dim_\Delta V'_i$ for $i=1,\dots, s$. One associates to
  $\varphi$ a tuple
\[ \bfn=(\dim_\Delta V_1,\dots, \dim_\Delta V_s) \]
which satisfies the condition $(\ref{eq:31})$ and
determines the map $\varphi$ up to equivalence. As such tuples are
in one-to-one correspondence with elements in $\calE_F(K,A)$. Then we
show a bijection map $e:A^\times \backslash \Hom_F^*(K,A)\to
\calE_F(K,A)$ which is given by
\begin{equation}
  \label{eq:37}
  e(\varphi)=(\dim_\Delta V_1/\ell_1, \dots, \dim_\Delta V_s/\ell_s).
\end{equation}
This completes the proof of the proposition. \qed
\end{proof}


\subsection{Global results}
\label{sec:32} In the remaining of this article
let $F$ be a global field. Let $A$ be
a central simple algebra over $F$ and $K$ a finite field extension
over $F$. We use the following notations.

\begin{itemize}

\item $A=\End_\Delta(V)$, where $\Delta$ is the division part of $A$,
  and $V$ is a finite right $\Delta$-module of rank $n$.

\item $k:=[K:F]$ and $\delta_0:=\deg(\Delta)$.

\item For any place $v$ of $F$, denote by $F_v$ the completion of $F$
  at $v$. Put
  \[ K_v:=K\otimes F_v=\prod_{w|v} K_w, \quad A_v:=A\otimes F_v, \quad
  \Delta_v=\Delta\otimes F_v=\Mat_{s_v}(D_v), \]
  where $D_v$ is the division part of the central simple algebra
  $\Delta_v$ (we do not use the letter $D$ as an algebra over $F$ in
  this section; do not confuse
  $D_v$ as the completion of $D$) and $s_v$ is the capacity
  of $\Delta_v$.

\item $k_w:=[K_w:F_v]$ and $d_v:=\deg(D_v)$,
  where $w$ is a place of $K$ over $v$.

\item $\Delta\otimes_F K=\Mat_c(\Delta')$ and $\delta':=\deg(\Delta')$,
  where $\Delta'$ is the
  division part of the central simple algebra $\Delta\otimes K$ over
  $K$, and $c$ is its capacity. One has
\begin{equation}
    \label{eq:38}
   \delta_0=\delta' c.
\end{equation}

\item For any place $w$ of $K$, put
\[ \Delta'_w:=\Delta'\otimes_K K_w=\Mat_{t_w}(D'_w), \quad
d'_w:=\deg(D'_w), \]
where $D'_w$ is the division part of the central simple algebra
$\Delta'_w$ and $t_w$ is the local capacity of $\Delta'$ at $w$.

\item $c_w:=\c(D_v\otimes_{F_v} K_w)$, i.e. $D_v\otimes_{F_v}
  K_w=\Mat_{c_w}(D'_w)$. One has
  \begin{equation}
    \label{eq:39}
    d_v=d_w' c_w.
  \end{equation}
It follows from
\begin{equation*}
  \begin{split}
   \Delta\otimes_F K_w &=(\Delta \otimes F_v)\otimes_{F_v} K_w=
 \Delta_v \otimes K_w=\Mat_{s_v c_w}(D'_w)\quad \text{and} \\
\Delta\otimes_F K_w & =(\Delta\otimes_F K)\otimes_K
K_w=\Mat_c(\Delta')\otimes_K K_w=\Mat_{c t_w}(D'_w)
  \end{split}
\end{equation*}
that
\begin{equation}
  \label{eq:310}
  s_v\, c_w=c\,t_w.
\end{equation}


\item For any rational number $a\in \Q$, we write $\bfd(a)$ for the positive
  denominator of $a$ in its reduced form, and $\bfn(a)$ for its
  numerator.

\item For each place $v$ of $F$, write
\[ \inv_v(\Delta)=\frac{a_v}{\delta_0}=\frac{ a'_v\, s_v}{d_v\,
  s_v}=\frac{a'_v}{d_v},\quad \gcd(a'_v,d_v)=1
  \text{\ and\ } s_v=\gcd(a_v,\delta_0). \]
  One has, by the Grunwald-Wang theorem
\begin{equation}\label{eq:311}
  \delta_0=\lcm \{d_v\}_{v\in V^F} \quad \text{and}\quad
  \left ( \gcd\{a_v\}_{v\in V^F},\delta_0 \right)=1,
\end{equation}
where $V^F$ denotes the set of all places of $F$.
\item For each place $w$ of $K$, write
\[ \inv_w(\Delta')=\frac{b_w}{\delta'}=\frac{{b'}_w\, t_w}{d'_w\,
  t_w}=\frac{{b'}_w}{d'_w},\quad \gcd({b'}_w,d'_w)=1
  \text{\ and\ } t_w=\gcd(b_w,\delta'). \]
  One has
\begin{equation}\label{eq:312}
  \delta'=\lcm \{d'_w\}_{w\in V^K} \quad \text{and}\quad
  \left ( \gcd\{b_w\}_{w\in V^K},\delta' \right)=1,
\end{equation}
where $V^K$ denotes the set of all places of $K$.

\item It follows from $\inv(D'_w)=\inv(D_v)[K_w:F_v]$ (see
  \cite{serre:lf}) that
  \begin{equation}
    \label{eq:313}
    c_w=\gcd(d_v,k_w).
  \end{equation} \\
\end{itemize}

Given $K$ and $A$, we have, for each place $v$ of $F$,
\begin{itemize}
\item a tuple
$(k_w)_{w|v}$ of positive integers,  and
\item a rational number $\inv_v(\Delta)=a'_v/d_v$
\end{itemize}
satisfying the following conditions:
\begin{itemize}
\item [(a)] $\sum_{w|v} k_w=k$ for all $v\in V^F$,
\item [(b)]
  \begin{itemize}
  \item [(i)] $d_v=1$ if $v$ is a complex place,
  \item [(ii)]$d_v\in\{1,2\}$ if $v$
  is a real place,
  \item [(iii)] $d_v=1$ for almost all $v$, and
  \item [(iv)] (Global class field theory) one has
\[  \sum_{v\in V^F} \frac{{a'}_v}{d_v}=0 \quad
  (\text{\ in\ } \Q/\Z). \]
  \end{itemize} \
\end{itemize}

We compute all other numerical invariants $\delta_0$, $c_w$, $d'_w$,
$\delta'$ and $c$ as follows.
\begin{itemize}
\item [(i)] The (global) degree $\delta_0$ of $\Delta$ can be
computed by (\ref{eq:311}).

\item [(ii)] Then one computes the local capacity $c_w$
of $D_v\otimes_{F_v} K_w$ and
the (local) degree $d'_w$ of $D'_w$ by
$(\ref{eq:313})$ and $(\ref{eq:39})$, respectively.

\item [(iii)]  Using (\ref{eq:312}) we compute
the (global) degree $\delta'$ of $\Delta'$ and then
compute the (global) capacity $c$ of $\Delta\otimes K$
using (\ref{eq:38}).
\end{itemize}

We define the following condition (G stands for ``global'') \\

\begin{itemize}
\item [(\bfG)]  $k\,|\,n\, c$.
\end{itemize}

\begin{prop}\label{33}
  The set $\Hom_F(K, A)$ is non-empty if and only if the
  condition $(\bfG)$ holds.
\end{prop}
\begin{proof}
  This follows from Lemma~\ref{23}. \qed
\end{proof}

Now we formulate the corresponding local conditions. Note that
\[ K_v=\prod_{w|v} K_w \quad \text{and} \quad A_v=\Mat_{n s_v}(D_v). \]
Put
\begin{equation}
  \label{eq:3135}
   \calE:=A^\times \backslash \Hom_F(K, A);
\end{equation}
the Noether-Skolem theorem says that if this set is non-empty then it
has one element.
For each place $v$ of $F$, define a set (c.f. (\ref{eq:35}))
\begin{equation}
  \label{eq:314}
\calE_v:=\calE_{F_v}(K_v,A_v)=\{(x_w)_{w|v}\,|\, x_w\in \bbN,\
\sum_{w|v} \ell_w x_w=n s_v \, \},
\end{equation}
where $\ell_w:=k_w/c_w$.
Define the following condition (L stands for ``local'') \\

\begin{itemize}
\item [(\bfL)] The set $\calE_v$ is non-empty for all $v\in V^F$.  \\
\end{itemize}

\begin{prop}\label{34}
  There is an embedding of $K_v$ in
$A_v$ over $F_v$ if and only if the set $\calE_v$ is non-empty.
\end{prop}
\begin{proof}
  This follows from Proposition~\ref{32}. \qed
\end{proof}
We have the following implication
\begin{center}
  the condition (\bfG) holds $\implies$ the condition (\bfL) holds.
\end{center}
The local-global principle then asks whether the converse is also true.

\subsection{Special vectors and the local-global principle}
\label{sec:33}
Let
\[ e_v: A_v^\times \backslash \Hom_{F_v}^*(K_v,A_v)\isoto \calE_v
\]
be the corresponding bijection obtained in Proposition~\ref{32}.
Let us suppose first that the set
$\Hom_F(K, A)=\Hom^*_F(K, A)$ of embeddings of
$K$ into $A$ over $F$ is {\it non-empty}.
For any element $\varphi$ in $\Hom_F(K,A)$,
let $\varphi_v\in \Hom_{F_v}(K_v,A_v)$ be the extension
of $\varphi$ by $F_v$-linearity,
and let $[\varphi_v]$ be its equivalence class.
Then one defines an element $\bfx_v\in \calE_v$ by
\[ \bfx_v:=e_v([\varphi_v]). \]
The association $\varphi\mapsto \bfx_v$ induces a well-defined map,
which we denote again by $e_v$,
\[ e_v:\calE\to \calE_v. \]
The non-emptiness of $\Hom_F(K,A)$ implies the existence of such a
vector $\bfx_v$ in $\calE_v$ for each place $v\in V^F$.
We now calculate these special vectors
explicitly.

The map $\varphi$ gives rise to a $(K,\Delta)$-bimodule structure on
$V$. Since $V$ is free $K$-module of rank $n\delta_0^2/k$, its
completion $V\otimes_F F_v$ is also a free $K_v$-module of same rank.
Therefore, one has
the decomposition
\[ V\otimes F_v=\bigoplus_{w|v} V_w,\]
where each factor $V_w$ is a $(K_w,\Mat_{s_v}(D_v))$-bimodule of
$K_w$-rank $n\delta_0^2/k$ (recall that $\Delta_v=\Mat_{s_v}(D_v)$).
Using the Morita equivalence, the module
 $V_w$ is isomorphic to $W_w^{\oplus s_v}$ for a
$(K_w,D_v)$-bimodule $W_w$ of $D_v$-rank $n s_v k_w/k$.
Using the formula (\ref{eq:37}), the
$w$-component $\bfx_w$ of the vector $\bfx_v$ is given by
\begin{equation}
  \label{eq:315}
  \bfx_w:=\dim_{D_v} W_w/\ell_w=n s_v c_w/k,
\end{equation}
which is a positive integer. Recall that $c_w=\gcd(k_w,d_v)$ and
$\ell_w=k_w/c_w$.

Therefore, this leads us to the following definition of
{\it special vectors} no matter the set $\calE$ is non-empty or not.
For each place $v$ of $F$ we
define a vector (still denoted by)
$\bfx_v=(\bfx_w)_{w|v}\in \prod_{w|v}\Q_{>0}$ by (\ref{eq:315}),
and we call them {\it special vectors}.
The above calculation shows if the set $\calE$ is non-empty, then the
vector $\bfx_v$ is the image of the map $e_v$.

\begin{prop}\label{35}
  Notations as above. If the set $\calE$ is non-empty, then one has
\[ \bfx_v\in \calE_v, \quad \forall\, v\in V^F, \]
or equivalently, each vector $\bfx_v$ lies in $\prod_{w|v} \bbN$ for all
$v\in V^F$.
\end{prop}

If we denote by $\bar \bfx_w$ the class of $\bfx_w$ in $\Q/\Z$, then
we associate to the pair $(K,A)$ an element
\begin{equation}
  \label{eq:316}
  \bar \bfx=(\bar \bfx_w)_{w\in V^K} \in \bigoplus_{w\in V^K}
\Q/\Z.
\end{equation}
Then Proposition~\ref{35} states that
the vanishing of the class $\bar \bfx$ is
a necessary
condition for the set $\calE$ to be non-empty. The following result
states that this is the only obstruction to the local-global
principle.

\begin{thm}\label{36} Notations as above. We have
\[ \Hom^*_F(K,A)\neq \emptyset \iff \bar \bfx=0. \]
\end{thm}
\begin{proof}
  Note that the condition $\bar \bfx=0$ implies $\bfx_v\in \calE_v$
  and hence $\Hom^*_{F_v}(K_v,A_v)$ is non-empty for all $v\in V^F$.
The implication $\implies$ is already proved.
To show the other direction, we must
  show that the condition (\bfG) $k\mid nc$ holds. Using
$c=\delta_0 / \delta'$ and $\delta'=\lcm\{d'_w\}$, we rewrite the
condition (\bfG) as
\begin{equation}
  \label{eq:318}
  k  \mid ({n\delta_0}/{d'_w} ),\quad \forall\, w \in V^K.
  \end{equation}

   Using (\ref{eq:39}) and (\ref{eq:315}), we have
  \[
  \bfx_w=n s_v c_w/k=n s_v d_v/k d'_w=n \delta_0 /k d'_w\in \bbN
  \]
  for all $w\in V^K$.
  This verifies the condition (\bfG) and hence proves the theorem. \qed
  \end{proof}

\subsection{An example.}
\label{sec:34}
We will show an example of a pair $(K,A)$,
where $K/F$ is a Galois extension and $A$ is a central simple
$F$-algebra so that
\begin{itemize}
\item the set $\Hom^*_{F_v}(K_v,A_v)$ is non-empty for all $v\in V^F$,
  and
\item the set $\Hom_F(K,A)$ is empty.
\end{itemize}
This particularly shows that the question (Q3) has a negative
answer. Let $K/F$ be a Galois
extension of degree $8$. Choose two places $v_1$ and $v_2$ of $F$ so
that
$k_w=2$ for all $w|v_1$ or $w|v_2$. Such places exist by the
Chebotarev density theorem. Let $A$ be a central simple
algebra over $F$ of degree $24$ that is ramified exactly at
the two places $v_1$ and $v_2$, and
\[ A_{v_1}=\Mat_{6}(D_{v_1}),\quad \text{and}\quad
A_{v_2}=\Mat_6(D_{v_2}), \]
where $D_{v_1}$ and $D_{v_2}$ are central division algebras of degree
$4$. The existence of such an $A$ follows from the global class field
theory. One has $c_w=2$ and $\ell_w=1$. The sets $\calE_{v}$ for
unramified places $v$ are non-empty. For $v=v_i$, $i=1,2$,
one see that
\[ \calE_{v_i}=\{(x_i)\in \bbN^4 |\sum_{i=1}^4 x_i=6\}\neq
\emptyset. \]
By Proposition~\ref{34},
the sets $\Hom^*_{F_v}(K_v,A_v)$ are non-empty for all $v\in V^F$.
On the other hand, $\bfx_w=6 c_w/k=12/8$
for any $w|v_1$ or $w|v_2$, which is not an
integer. By Theorem~\ref{36}, the set $\Hom_{F}(K,A)$ is empty.

The same argument shows that there is a Galois extension $K/F$ of
degree $p^m$, and a central simple algebra $A/F$ of degree $p^m q$,
where $p$ and $q$ are primes,  with
$m \ge 2$ and $p<q$, so that the local-global principle for
embedding $K$ into $A$ fails. One takes $k_w=p$ and $d_v=p^2$.
Then $c_w=p$ and $\ell_w=1$, We see that
$\calE_{v_i}$ is non-empty from the inequality $p^{m-1}\le p^{m-2}q$.
However, $\bfx_w=p^{m-2}q p/p^m\notin \bbN$.






\section{The local-global principle}
\label{sec:04} Keep the notation in \S~\ref{sec:32} and \ref{sec:33}.
In this section,
we study on the local-global principle in detail.

For the convenience of discussion, we make the following definitions.

\begin{defn}\label{41}\

\begin{itemize}
\item[(1)] Let $K$ be a finite field extension over a global field $F$
and $A$ be a central simple algebra over $F$. We say that the
Hasse principle for the pair $(K,A)$ holds if the following properties
are equivalent:
\[
\Hom^*_{F_v} (K_v, A_v)\neq \emptyset,\quad  \forall\, v\in V^F
\Longleftrightarrow  \Hom^*_{F} (K, A)=\Hom_{F}(K,A) \neq   \emptyset.
\]
In other words, the conditions (\bfG) and (\bfL) in \S~\ref{sec:32} are
equivalent.


\item[(2)] Let $F$ be a global field.
Let $(k,\delta)$ be a pair of two positive integers with
$k\mid \delta$.
We say that the Hasse principle for the pair $(k,\delta)$ holds
if for any finite field extension $K$ over $F$ of degree $k$
and for any central simple algebra $A$ over $F$
of degree $\deg(A)=\delta$,
the Hasse principle for the pair $(K,A)$ holds.
\end{itemize}
\end{defn}

Theorem~\ref{36} gives us an effective way to check the Hasse
principle. Namely, one only needs to check
\begin{equation}
  \label{eq:41}
  \bfx_w=ns_vc_w/k\stackrel{?}{\in} \bbN, \quad  \forall\, w\in V^K
\end{equation}
for a given pair $(K,A)$.


\subsection{Basic positive results}
\label{sec:41}

\begin{prop}\label{42}
Let $A$ be a central simple algebra over $F$ with index $
\i(A)=\delta_0$ and degree $\deg A= n\delta_0$, and let
$K$ be a finite field extension over $F$ of degree
$[K:F]\mid \deg A$.
Suppose one of the following properties holds:
\begin{itemize}
\item[(1)]  $[K:F]=\deg A$.

\item[(2)] $K$ splits $A$.

\item[(3)] For any place $v$, the algebra $A_v$ is either a
 division algebra or a matrix
 algebra over $F_v$. 
\end{itemize}
Then the Hasse principle for the pair $(K,A)$ holds.
\end{prop}

\begin{proof}

\begin{itemize}
\item [(1)] This is a well-known result;
see Pierce~\cite[\S~18.4]{pierce} or
Prasad-Rapinchuk \cite[Proposition
A.1]{prasad-rapinchuk:metakernel96}.


\item [(2)] Since  $K$ splits $A$, we have $d'_w=1$ for all $w\in V^K$
  and hence $\delta'=\lcm\limits_w
\{d'_w\}=1$ and $c=\delta_0$.
Thus, \[k\mid n\delta_0\Longleftrightarrow k\mid nc.\]
This verifies the condition (\bfG) and hence $\Hom_F(K,A)\neq \emptyset$.

\item [(3)] Suppose there exists an embedding of $K_v$ into $A_v$
  for all $v \in V^F$.
If $A_v$ is a central division algebra, then
$K_v=K_w$ is a field extension over $F_v$ with $[K_w:F_v]=k$ and
$c_w=k$. Then we have
\[
\bfx_w=ns_vc_w/k=n\in \bbN,\quad  \forall\, w \mid v.
\]
For the other case that $A_v$ is a matrix algebra over $F_v$,
we have $s_v=\delta_0$. Thus, by the assumption $k\mid
n\delta_0$, we have
\[
\bfx_w=ns_vc_w/k=n\delta_0 c_w/k \in \bbN,\quad  \forall\, w \mid v.
\]
By Theorem~\ref{36}, the set $\Hom_F(K,A)$ is not empty. \qed
\end{itemize}
\end{proof}



\subsection{Construction of counterexamples}
\label{sec:42}

Let $K$ be a finite field extension over $F$ with
$[K:F]=k=p_1^{m_1}\cdots p_r^{m_r}$, and $A$ be a central division
algebra over $F$ of degree $\delta=p_1^{n_1}\cdots p_r^{n_r}$, where
$p_i$ is a prime number, $n_i\in \bbN$, and $m_i\in \bbZ_{\geq 0}$
with $m_i\leq n_i$ for $i=1,\ldots, r$.


We shall construct examples for which the Hasse principle does
not hold. Recall (Definition~\ref{41}) that the Hasse
principle for a pair $(k,\delta)$ does not hold if there exists a
pair $(K,A)$ such that $\Hom^*_{F_v}(K_v, A_v)\neq \emptyset$
for all $v\in V^F$ but $\Hom_F(K,A)=\emptyset$.

The construction is to set local data $s_v$,
$d_v$, and a partition 
$k=\sum_{w\mid v} k_w$ of the integer
$k$ for some place $v$ of $F$
so that some $w$-component $\bfx_w$ of the special vector
$\bfx_v$ is not integral. Then we apply Theorem~\ref{36} to conclude
that the Hasse principle for certain pair $(K,A)$ does not hold.
To ensure that the data $k=\sum_{w\mid v} k_w$
come from a global field $K$, we
need the following result.


\begin{lemma}[c.f. {\cite[Lemma 3.2]{yu:const}}]\label{43}
Let S be a finite subset of $V^F$.
Let $L_v$, for each $v\in S$, be any etale algebra over
$F_v$ of same degree $[L_v:F_v]=n$.
Then there is a finite separable field extension $K$
over $F$ of degree $n$ such that $K \otimes_F F_v\simeq L_v$ for all
$v \in S$.
\end{lemma}
Moreover, one also needs to ensure that the data $s_v$ and $d_v$
come from a central division algebra $A$ over $F$. For this, we
need the global class field theory; see \cite[\S~32]{reiner:mo} or
\cite{cassels-frohlich:ant}.

\begin{thm}\label{44}
Let $S$ be a finite subset of $V^F$.
For any positive integer $\delta$,
suppose we are given any set of rational numbers
$\{a_v/d_v\}_{v\in S}$ with $\gcd(a_v,d_v)=1$ such that

\begin{itemize}

\item [(1)] $\lcm\limits_{v\in S} \{d_v\}=\delta$.

\item [(2)] $\sum\limits_{v\in S} a_v/d_v=0\in \bbQ/\bbZ$.

\item [(3)] $d_v=1$ if $v$ is complex.

\item [(4)] $d_v=1$ or $2$ if $v$ is real.
\end{itemize}

Then up to isomorphism, there is a unique central division algebra
$A$ over $F$ with $\deg A=\delta$ such that $\inv
A_v=a_v/d_v\in\bbQ/\bbZ$ for all $v\in S$ and $\inv A_v=0\in \bbQ/\bbZ$
for all $v\notin S$.
\end{thm}

Let $(k,\delta)$ be a pair of positive integers with $k\mid \delta$.
For any partition $\lambda$ of $k$
\[ \lambda=(k_1,\dots, k_t), \quad k_1\le \dots \le k_t,
\quad \sum_{i=1}^t k_i=k \]
and a positive integer $s$ dividing $\delta$,
we set
\[ \calE(\lambda,s):=\{(x_i)\in \bbN^t \mid \sum_{i=1}^t \ell_i
x_i=s\}, \]
where $\ell_i:=k_i/\gcd(k_i,d)$,
and define a vector $\bfx(\lambda,s)\in \Q^t$, where
\[ \bfx(\lambda, s):=(\bfx_i), \quad
\bfx_i=\bfx_i(\lambda,s):=s \gcd(k_i,d)/k, \quad \forall\,
i=1\dots, t. \]
Let $\LD (k,\delta)$ be the set of pairs $(\lambda, s)$ of
partitions $\lambda$ of $k$ and divisors $s$ of $\delta$ such
that $\calE(\lambda,s)$ is non-empty (``\LD'' stands for all
possible local decompositions).

We transform the problem of checking the Hasse principle for the pair
$(k,\delta)$ in purely combinatorial terms.

\begin{thm}\label{45}
Let $F$ be a global field and
let $(k,\delta)$ be a pair of positive integers with $k\mid\delta$.
\begin{itemize}
\item [(1)] Suppose for all elements $(\lambda,s)\in \LD(k,\delta)$,
  the vector $\bfx(\lambda,s)$ has integral components. Then the
  Hasse principle for $(k,\delta)$ holds.
\item [(2)] Suppose there is an element $(\lambda,s) \in
  \LD(k,\delta)$ such that $\bfx_i(\lambda,s)\not\in \bbN$ for some
  $i$. Then there are a finite separable field extension $K$ of degree
  $k$ over $F$ and a central division algebra $A$ of degree
  $\delta$ over $F$ so that the
  Hasse principle for $(K,A)$ fails. In particular, the Hasse
  principle for the pair $(k,\delta)$ fails.
\end{itemize}
\end{thm}

\begin{proof}
(1) Let $K$ be any finite field extension of $F$ of degree $k$ and $A$
be any central simple algebra over $F$ of degree $\delta$. For any
place $v$ of $F$, we get a partition $\lambda=(k_w)_{w|v}$ of $k$ and
a positive divisor $s$ of $\delta$ as the capacity of
$A_v=A\otimes F_v$. 
The assumption
$\Hom^*_{F_v}(K_v,A_v)\neq \emptyset$ assures that
$\calE(\lambda,s)=\calE_v$ is non-empty. Therefore, the pair
$(\lambda,s)$ is an element in $\LD(k,\delta)$. Then the assumption
gives that $\bfx_v=\bfx(\lambda,s)$ has integral components. This
works for all $v\in V^F$. By Theorem~\ref{36},
we have $\Hom^*_F(K,A)\neq \emptyset$.

(2) Let $(\lambda,s)\in \LD(k,\delta)$,
where $\lambda=(k_1,\dots, k_t)$ and $s|\delta$,
be an element such that
$\bfx(\lambda,s)$ is not an integral vector.
Choose any finite set $S=\{v_1,v'_1,v_2,v'_2\}$ of 4 finite places of
$F$. Let $A$ be the central division algebra over $F$ of degree
$\delta$ with following local invariants
(Theorem~\ref{44}):
\begin{itemize}
\item $\inv_{v_1} (A)=-\inv_{v_1'}(A)=1/d$ and
$\inv_{v_2} (A)=-\inv_{v'_2} (A)= 1/\delta$.
\item $\inv_v(A)=0$ for all $v\not\in S$.
\end{itemize}
Let $L_{v_2}$ and $L_{v_2'}$ be any finite separable {\it field}
extensions of degree $k$ over $F_{v_2}$ and $F_{v_2'}$,
respectively. Let $L_{v_1}=\prod_{i=1}^t E_i$ and
$L_{v'_1}=\prod_{i=1}^t E'_i$ where $E_i$ (resp. $E'_i$) is a
separable field extension of $F_{v_1}$ (resp. of $F_{v'_1}$) of
degree $k_i$ for all $i$. By Lemma~\ref{43}, there exists a finite
separable field extension $K$ of $F$ of degree $k$ such that
$K\otimes F_v\simeq L_v$ for all $v\in S$. It follows from
$(\lambda,s)\in \LD(k,\delta)$ that $\calE_{v_1}$ and
$\calE_{v'_1}$ are non-empty. Also the sets $\calE_{v_2}$ and
$\calE_{v_2'}$ are non-empty as $K_{v_2}$ and $K_{v'_2}$ are fields.
Since $A_v$ is a matrix algebra for $v\notin S$, these sets
$\calE_v$ are non-empty, too. Finally, since $\bfx(\lambda,s)$ is
not integral, the special vector $\bfx_{v_1}=\bfx(\lambda,s)$
is not integral.
Therefore, we have constructed a pair $(K,A)$ for which the Hasse
principle fails. \qed
\end{proof}

\begin{prop}\label{46}
Let $\delta=p_1^{n_1}\cdots p_r^{n_r}$, and let $k$ be a positive integer
with $k>1$ and $k \mid \delta$, where $p_i$\!\!'s are prime divisors
of $\delta$.
Assume that $k \leq \delta/p_i^{n_i}$
and $p_i \mid k$ for some $1\leq i\leq r$ (so $r\ge 2$).
Then the Hasse principle
for the pair $(k,\delta)$ does not hold.
\end{prop}

\begin{proof}
Without loss of generality, we may assume $i=1$. We write
$k=p_1^{m_1}k'$ for some prime-to-$p_1$ integer $k'$.
Let
\begin{itemize}
\item[(i)] $s=\delta/p_1^{n_1}=p_2^{n_2}\cdots p_r^{n_r}$.
\item[(ii)] $\lambda=(1,\dots,1)=:(1^{\oplus k})$.

\end{itemize}
Then one has
\[
d=p_1^{n_1},\quad c_i:=\gcd(k_i, d)=1, \quad \ell_i:=k_i/c_i=1, \]
and
\[ \bfx_i=s/k=p_2^{n_2}\cdots p_r^{n_r}/p_1^{m_1}k'\notin \bbN,\quad
\forall i. \]
Note that the condition $k \leq\delta/p_1^{n_1}=s$ implies the equation
$X_1+\cdots+X_k=s$ has a positive integral solution,
i.e. the set $\calE(\lambda,s)$ is non-empty.
By Theorem~\ref{45}, the Hasse principle does not hold. \qed
\end{proof}

\begin{cor}\label{47}
  Let $\delta$ be a positive integer divisible by at least
  two primes. Then
  there is a positive integer $k$ with $k\mid \delta$ such that the
  Hasse principle for the pair $(k,\delta)$ does not hold.
\end{cor}
\begin{proof}
  Write $\delta=p_1^{n_1}\cdots p_r^{n_r}$, $r\ge 2$, and assume that $p_1$
  is the smallest prime divisor. Let $k=p_1$. Then the corollary
  follows from Proposition~\ref{46}. \qed
\end{proof}

\begin{cor}\label{48}
   Let $\delta$ be a positive integer divisible by at least
   two primes. Then
   there are a central simple algebra $A$ of degree $\delta$ and a
   finite field extension of $F$ of degree $[K:F]\mid\delta$ so that the
   Hasse principle for the pair $(K,A)$ does not hold.
\end{cor}








\section{Hasse principle for homogeneous spaces}
\label{sec:07}

The space of all embeddings of an etale commutative algebra $K$ into
$A$ over $F$ forms an variety $X$ over $F$. It is naturally equipped
with an action by the algebraic group $A^\times$ through
conjugation. Therefore, the space $X$ is decomposed into several orbits
and each orbit is a homogeneous space of $A^\times$.  
In the previous section, we see some counterexamples for the Hasse
principle in a more combinatorial way. A closer look at these
counterexamples would find out that some of local points actually 
lie in different orbits. This is the main reason 
that causes the failure of the Hasse
principle. It turns out this is the only reason that fails the Hasse
principle. Namely, 
if we requires that all local points lie in the same orbit $Y$, then
$Y$ indeed has an $F$-rational point. This the main result of this section;
see Theorem~\ref{74}.    
  


\subsection{$F$-kernels and $H^2$}
\label{sec:71}
We recall the definition and basic properties of the second
non-abelian Galois cohomology $H^2$. Our references are Springer
\cite{springer:h2} and Borovoi~\cite{borovoi:duke1993}.

Let $F$ be a field and let $F^\sep$ denote a separable closure of $F$.
Put $\Gamma_F:\Gal(F^\sep/F)$.   
Let $\ol G$ be a linear algebraic group over $F^\sep$ and
let $p: \ol G\to \Spec F^\sep$ be the structure morphism.

For $\sigma\in
\Gamma_F$, let $\sigma_\natural:\Spec F^\sep\to \Spec F^\sep$
denote the isomorphism induced by the map $\sigma^{-1}$. We have $(\sigma
\tau)_\natural=\sigma_\natural \tau_\natural$ for $\sigma,\tau\in
\Gamma_{F}$.

Let $\sigma\in \Gamma_F$. We say that an automorphism of schemes
$s:\ol G \to \ol G$ is a {\it $\sigma$-semialgebraic automorphism
of $\ol G$} if the following
diagram
\[
\begin{CD}
  \ol G  @>{s}>> \ol G \\
   @V{p}VV  @VV{p}V \\
  \Spec F^\sep @>{\sigma_\natural}>> \Spec F^\sep
\end{CD} \]
commutes and the induced morphism $s': \ol G \to \ol G^\sigma$
is an isomorphism of algebraic groups over $\Spec F^\sep$, where $\ol
G^\sigma:=\ol G\times_{\sigma_\natural} \Spec F^\sep$ is the base
change deduced by the morphism $\sigma_\natural$. An {\it semialgebraic
automorphism} $s$ of $\ol G$ is a $\sigma$-semialgebraic automorphism
of $\ol G$ for some $\sigma\in \Gamma_F$. Define $\gamma(s):=\sigma$
if $s$ is a $\sigma$-semialgebraic automorphism of $\ol G$.

\def\SAut{{\rm SAut\, }}
\def\Out{{\rm Out\, }}
\def\SOut{{\rm SOut\, }}
Let $\SAut \ol G$ denote the group of semialgebraic automorphisms of
$\ol G$. The map $\gamma: \SAut \ol G\to \Gamma_F$ is a
homomorphism. Let $\Aut \ol G$ denote the group of algebraic
automorphisms of $\ol G$ over $F^\sep$. We have an embedding $\Aut \ol
G \to \SAut \ol G$ and an exact sequence
\[
\begin{CD}
 1 @>>> \Aut \ol G @>>> \SAut \ol G @>{\gamma}>> \Gamma_F.
\end{CD} \]
Let $\Int \ol G$ denote the group of inner automorphisms of $\ol
G$. Set $\Out \ol G:=\Aut \ol G/\Int \ol G$ and
$\SOut \ol G:=\SAut \ol G/\Int \ol G$. We have an exact sequence

\begin{equation}
  \label{eq:71}
  \begin{CD}
 1 @>>> \Out \ol G @>>> \SOut \ol G @>{q}>> \Gamma_F.
  \end{CD}
\end{equation}

\begin{defn}
  An {\it $F$-kernel} is a pair $L=(\ol G,\kappa)$, where $\ol G$ is an
  $F^\sep$-group and $\kappa:\Gamma_F\to \SOut \ol G$ is a
  homomorphism satisfying the following two conditions:
  \begin{itemize}
  \item [(a)] $\kappa$ is a splitting of (\ref{eq:71});
  \item [(b)] $\kappa$ can be lifted to a continuous map
    $f:\Gamma_F\to \SAut \ol G$; here ``continuous'' means that the
    stabilizer in $\Gamma_F$ of any regular function $\phi\in
    F^\sep[\ol G]$ is open.
  \end{itemize}
\end{defn}

For an $F^\sep$-kernel $L=(\ol G,\kappa)$, the second Galois
cohomology set $H^2(F, L)=H^2(F, \ol G, \kappa)$ is defined as
follows (also see \cite[p.~221]{borovoi:duke1993}). A 2-cocycle is a pair
$(f,u)$ of continuous maps
\[ f:\Gamma_F\to \SAut \ol G, \quad u:\Gamma_F\times \Gamma_F \to \ol
G(F^\sep) \]
such that for any $\sigma, \tau, v\in \Gamma_F$, the following holds:
\begin{equation}
  \label{eq:72}
  \begin{split}
    {\rm int}(u_{\sigma,\tau})\circ f_\sigma\circ f_\tau &=f_{\sigma
    \tau},\\
    u_{\sigma,\tau v}\cdot f_\sigma(u_{\tau,v})&=u_{\sigma\tau, v}\cdot
    u_{\sigma,\tau}, \\
    f_\sigma\!\!\! \mod \Int \ol G& =\kappa(\sigma).
  \end{split}
\end{equation}

Let $Z^2(F,L)$ denote the set of 2-cocycles with coefficients in
$L$. The group $C(F,\ol G)$ of continuous maps $c:\Gamma_F\to \ol
G(F^\sep)$ acts on $Z^2(F,L)$ on the left by
\[ c \cdot (f,u)=(f',u') \]
where
\begin{equation}
  \label{eq:73}
  f_\sigma'={\rm int}(c_\sigma)\circ f_\sigma, \quad
  u'_{\sigma,\tau}=c_{\sigma\tau}\cdot u_{\sigma,\tau}\cdot
  f_\sigma(c_\tau)^{-1} \cdot c_\sigma^{-1}.
\end{equation}
The quotient set $H^2(F,L):=C(F,\ol G)\backslash Z^2(F,L)$ is called
the {\it second Galois cohomology set of $F$ with coefficients in $L$}. If
$(f,u)\in Z^2(F,L)$, we write $\Cl(f,u)$ for the cohomology class of
$(f,u)$ in $H^2(F,L)$. A {\it neutral 2-cocycle} is a cocycle of the
form $(f,1)$. A {\it neutral cohomological class} in $H^2(F,L)$ is the
class of a neutral cocycle.

Let $G$ be an $F$-group. Set $\ol G:=G\otimes_F F^\sep$. For
$\sigma\in \Gamma_F$, let $\sigma_*:\ol G=\ol G^\sigma \to \ol G$ be the
morphism induced by the fiber product
\[
\begin{CD}
  \ol G  @>{\sigma_*}>> \ol G \\
   @V{p}VV  @VV{p}V \\
  \Spec F^\sep @>{\sigma_\natural}>> \Spec F^\sep
\end{CD} \]
Then $\sigma_*$ is a $\sigma$-semialgebraic automorphism of $\ol
G$. We have $(\sigma\tau)_*=\sigma_* \tau_*$ for $\sigma,\tau\in \Gamma_F$. We obtain
a continuous homomorphism
\[ f_G: \Gamma_F\to \SAut \ol G, \quad \sigma\mapsto \sigma_* \]
which splits the exact sequence (\ref{eq:71}). Composing $f_G$ with
the homomorphism $\SAut \ol G \to \SOut \ol G$, we obtain a
homomorphism
\[ \kappa_G:\Gamma_F \to \SOut \ol G, \quad
\kappa_G(\sigma)=\sigma_*\!\!\! \mod \Int \ol G \]
and, thus, an $F$-kernel $L_G=(\ol G, \kappa_G)$. Set
$H^2(F,G):=H^2(F,L_G)$. When $G$ is commutative, $H^2(F,G)$ is the
group cohomology $H^2(\Gamma_F, G(F^\sep))$ and it is an abelian
group. In $H^2(F,G)$, we have a neutral class
$\Cl(f_G,1)$, which we denote by $n(G)$.

Suppose $H^2(F,L)$ of an $F$-kernel $L=(\ol G,\kappa)$ contains a
neutral class $\Cl(f,1)$. Then by Weil's descent theorem
there is an $F$-form $G$ of $\ol G$, unique up to $F$-isomorphism,
 such that $f=f_G$ and $\kappa=\kappa_G$. Let
\[ \psi:\Gamma_F \to (G/Z)(F^\sep)=\Int \ol G \]
be a cocycle, where $Z$ is the center of $G$.
The map $\psi$ defines an inner
form $G'=_\psi G$ of $G$ and gives rise to
a homomorphism $f':=f_{G'}$.
We have $f'_\sigma=\psi_\sigma
f_\sigma$ for all $\sigma\in \Gamma_F$ and $\kappa_{G'}=\kappa$.
This defines another neutral class $n(G')=\Cl(f',1)$ in
$H^2(F,L)$. Conversely, any neutral class arises in this way. Indeed,
one puts $\psi_\sigma:=f'_\sigma f_\sigma^{-1}$ and then the map
$\psi:\Gamma_F\to
\Int \ol G$ is a cocycle.

\subsection{Homogeneous spaces}
\label{sec:72}

Let $H$ be a connected reductive group over $F$. Let $X$ be a right
homogeneous space of $H$. The homogeneous space $X$ gives rise to an
$F$-kernel $L$ as follows. Let $x\in X(F^\sep)$ be an
$F^\sep$-point and let $\ol G$ be the stabilizer subgroup
of $x$ in $\ol H:=H\otimes_F F^\sep$. For $\sigma\in \Gamma_F$, write
\[ ^\sigma x=x\cdot h_\sigma   \]
where $h_\sigma\in H(F^\sep)$. As $\Gamma_F$ acts continuously on the
discrete set $X(F^\sep)$, we have a continuous map $h:\Gamma_F\to
H(F^\sep)$. The $\sigma$-semialgebraic automorphism of $\ol H$
\[ f_\sigma:={\rm int}(h_\sigma)\circ \sigma_* \]
takes $\ol G$ into itself. To see this, if $g\in \ol G$, then we have
$x\cdot g=x$ and $^\sigma\! x\cdot {}^\sigma\! g= {}^\sigma\! x$. So
$x\cdot
h_\sigma {}^\sigma\! g h_\sigma^{-1}=x$ and hence
$\kappa(\sigma)(g)=h_\sigma {}^\sigma\! g h_\sigma^{-1}\in \ol G$.
We regard $f_\sigma$ as a
$\sigma$-semialgebraic automorphism of $\ol G$.
Then $f:\Gamma_F\to \SAut \ol G$ is a a continuous map, the composition
\[ \kappa: \Gamma_F\to \SAut \ol G \to \SOut \ol G \]
is a homomorphism and this defines an $F$-kernel $L=(\ol G, \kappa)$.

A {\it principal homogeneous space of $H$ over $X$} is a pair
$(P,\alpha)$, where $P$ is a right principal homogeneous space of $H$
and $\alpha:P\to X$ is an $H$-equivariant $F$-morphism. The etale descent
shows that $\alpha$ is necessarily faithfully flat.

If $X$ has an $F$-rational point $x_0$, then there exists a principal
homogeneous space $(P,\alpha)$ over $X$. Indeed, one takes $P=H$ and
$\alpha(h)=x_0\cdot h$. Conversely, if there exists a principal
homogeneous space $(P,\alpha)$ over $X$ and $P$ has an $F$-rational
point $p_0$, then $X$ has an $F$-rational point $\alpha(p_0)$.

In \cite[1.20]{springer:h2} (also see \cite[7.7]{borovoi:duke1993}),
Springer defined a cohomology class $\eta(X)\in
H^2(F, L)$ associated to $X$.
Springer \cite[3.7]{springer:h2} and Borovoi \cite[7.7,
p.~235]{borovoi:duke1993}
proved the following result.

\begin{prop}\label{72}
  Notations as above. The class $\eta(X)$ is neutral if and only if
  there exists a principal homogeneous space $(P,\alpha)$ over $X$.
\end{prop}

In particular, if $H^1(F,H)$ is trivial, then $\eta(X)$ is neutral if
and only if $X(F)$ is non-empty.

\subsection{Homogeneous spaces of embeddings of $K$ into $A$}
\label{sec:73} We return to the problem of embeddings of an etale
  algebra $K$ into a central simple algebra $A$. We remain $F$ a
  global field as in the previous sections. 
Let $X$ be the $F$-scheme that represents the functor
\[ X(R)=\Hom^*_{R-alg}(K\otimes_F R, A\otimes_F R) \]
for any commutative $F$-algebra $R$. The group $H=A^\times$ of
multiplicative group of $A$, viewed as an algebraic group over $F$,
acts naturally on $X$ on the left. We make the right $H$-action on $X$ by
setting $x\cdot h:=h^{-1} \cdot x$.

The geometric orbits of $X$ under $H$
are in one-to-one correspondence with the
elements in the orbit set
$\calO_{K\otimes F^\sep, A\otimes F^\sep}$. Therefore each geometric orbit
can be represented by a function $f:\Sigma\to \bbN$ on $\Sigma$, where 
$\Sigma:=\Hom_F(K, F^\sep)$. 
Moreover, this correspondence is
  $\Gamma_F$-equivariant. We have the decomposition of geometric orbits
\[ \ol X:=X\otimes_F F^\sep=\bigcup_{f} \ol X_f,  \]
where each $\ol X_f$ is a homogeneous space of $\ol H:=H\otimes_F
F^\sep$ corresponding to the function $f$.
As the correspondence $f\mapsto \ol X_f$ is $\Gamma_F$-equivariant,
the subvariety $\ol X_f$ is defined over $F$ 
if and only if the function $f$ is
$\Gamma_F$-invariant.

Let $\ol X_f$ be a geometric orbit which is defined over $F$. 
Let $X_f$ be the $F$-subscheme of $X$ whose base extension to $F^\sep$
is $\ol X_f$; this is a homogeneous space of
$H$ over $F$. As $f$ is $\Gamma_F$-invariant, 
there is a tuple $(m_1,\dots,
m_r)\in \bbN^r$ so that for any point $\ol \epsilon\in X_f(F^\sep)$
\[ \ol \epsilon:
K\otimes_F F^\sep \hookrightarrow \ol A:=A\otimes
F^\sep=\End_{F^\sep}(\ol V), \]
the induced $K\otimes F^\sep$-module $\ol V$ is
isomorphic to
\[ (K_1\otimes F^\sep)^{m_1}\oplus \dots \oplus (K_r\otimes
F^\sep)^{m_r}. \]
Fix an element $\ol \epsilon_0$ in $X_f(F^\sep)$.
The centralizer $\ol B$ of the subalgebra
$\ol \epsilon_0(K\otimes F^\sep)$ in $\ol A$ is
isomorphic to
\[ \Mat_{m_1}(K_1\otimes F^\sep)\times \dots \times
\Mat_{m_r}(K_r\otimes F^\sep), \]
and the stabilizer subgroup $\ol G$ of $\ol \epsilon_0$ is equal to the
multiplicative group $\ol B^\times$ (viewed as an algebraic group over
$F^\sep$).
Let $\ol Z$ be the center of $\ol G$.
The center $Z(\ol B)$ of $\ol
B$ is equal to $\ol \epsilon_0(K\otimes F^\sep)$ and we have $\ol
Z=Z(\ol B)^\times$.
The homomorphism
$\kappa:\Gamma_F\to \SOut \ol Z=\SAut \ol Z$
associated to $X_f$ gives an action on the center
\begin{equation}
  \label{eq:74}
  \ol Z(F^\sep)=Z(\ol B)^\times=\ol \epsilon_0(K\otimes
  F^\sep)^\times.
\end{equation}

\begin{lemma}\label{73} \
  \begin{enumerate}
  \item The center $Z$ of $(\ol G,\kappa)$ is isomorphic to
    $\prod_{i=1}^r \Res_{K_i/F} \Gm$.
  \item Any $F$-form $G$ of $\ol G$ with
$\kappa_G=\kappa$ is of the form
$\GL_1(A_1) \times \dots \times \GL_1(A_r)$
for some central simple $K_i$-algebras $A_i$ of degree $m_i$.
\end{enumerate}
\end{lemma}
\begin{proof}
(1) For any $\sigma\in \Gamma_F$, we have $^\sigma \ol
    \epsilon_0={\rm int}(h_\sigma^{-1}) \circ \ol \epsilon_0$
   for some element $h_\sigma \in \ol A^\times$. We have
   the commutative diagram:
\begin{equation}
   \label{eq:75}
   \begin{CD}
   K\otimes F^\sep @>{\ol \epsilon_0}>> A\otimes F^\sep \\
   @V{{\rm id}_K\otimes \sigma} VV @VV{{\rm id}_A\otimes \sigma}V \\
   K\otimes F^\sep @>{^\sigma \ol \epsilon_0}>> A\otimes F^\sep.
   \end{CD}
\end{equation}
As $\kappa(\alpha)={\rm int}(h_\sigma)\circ ({\rm id}_A\otimes
\sigma)$, we have the following commutative diagram:
\begin{equation}
   \label{eq:76}
   \begin{CD}
   K\otimes F^\sep @>{\ol \epsilon_0}>> A\otimes F^\sep \\
   @V{{\rm id}_K\otimes \sigma} VV @VV{\kappa(\sigma)}V \\
   K\otimes F^\sep @>{ \ol \epsilon_0}>> A\otimes F^\sep.
   \end{CD}
\end{equation}
This shows that $\kappa(\sigma)$ acts on $\ol \epsilon_0(K\otimes
F^\sep)$ on the second factor and hence $Z(F)=\ol \epsilon_0(K)^\times$.
The $F$-subalgebra generated by $Z(F)$ is equal to $\ol
\epsilon_0(K)$. Therefore, $Z$ is the subtorus $\ol
\epsilon_0(K)^\times$ and is isomorphic to
    $\prod_{i=1}^r \Res_{K_i/F} \Gm$. .

(2) Since the $Z(F)$ generates the $F$-subalgebra $\ol \epsilon_0(K)$, 
the $F$-subalgebra generated by $G(F)$ in $\ol A$ is a semi-simple
$F$-algebra containing $\ol \epsilon_0(K)$, which is of the 
form $A_1\times \dots \times A_r$, 
where each $A_i$ is a central simple algebra over
$\ol \epsilon_0(K_i)$. 
Therefore, $G$ is isomorphic to 
$\GL_1(A_1) \times \dots \times \GL_1(A_r)$. 
This completes the proof of the lemma. \qed
\end{proof}


\begin{thm}[The Hasse principle]\label{74}
  Let $X_f$ be a homogeneous space  of $H$ defined over $F$ and
  $(m_1,\dots,m_r)$ be the corresponding tuple as above. If
  $X_f(F_v)\neq \emptyset$ for all $v\in V^F$, then
  $X_f(F)\neq \emptyset$.
\end{thm}
\begin{proof}
  Since $H$ is an inner form of $\GL_m$, one has $H^1(F,H)=1$ and
  $H^1(F_v,H)=1$ by the Hilbert 90 theorem. By Proposition~\ref{72},
  the obstruction class $\eta(X_f)$ (resp.~$\eta_v(X_f)$) is neutral
  if and only if $X_f(F)\neq \emptyset$ (resp.~$X_f(F_v)\neq
  \emptyset$). Choose the quasi-split form
  $G_0=\prod_{i=1}^r \Res_{K_i/F} \GL_{m_i}$ of the isotropic
  subgroup $\ol G$ with $\kappa_{G_0}=\kappa$,
  and let $n(G_0)$ denote the neutral class
  $\Cl(f_{G_0},1)\in H^2(F,L)$. As $H^2(F,L)$ is non-empty, it is a
  principal homogeneous space under $H^2(F,Z)$; see \cite[1.17,
  p.~170]{springer:h2} . The base class
  $n(G_0)$ gives a natural bijection $H^2(F,Z)\simeq H^2(F,L)$. As
  $H^2(F,Z)=\Br(K_1)\times \dots \times \Br(K_r)$,
  the class $\eta(X_f)\in
  H^2(F,L)$ can be represented by $([A_1],\dots , [A_r])\in
  H^2(F,Z)$, where each $A_i$ is a central simple algebra over $K_i$.
  Moreover, the class $\eta(X)$ is neutral if and only if the class
  $([A_1],\dots ,[A_r])$ lies in the image of the boundary map
  $H^1(F, G_0/Z)\to
  H^1(F,Z)$ (see \cite[Proposition 2.3, p.~224]{borovoi:duke1993}),
  that is, the condition
\[ {\rm ind}_{K_i}(A_i)|m_i, \quad i=1,\dots
  ,r \]
holds.
Now under the assumption $X_f(F_v)\neq \emptyset$ for all $v\in
  V^F$ we get the condition
\[ {\rm ind}_{K_{i,w}}(A_i\otimes_{K_i}
  K_{i,w})|m_i\]
for all $i=1,\dots, r$ and $w\in V^{K_i}$. It follows from the
Hasse-Brauer-Noether theorem that ${\rm ind}_{K_i}(A_i)|m_i$, 
for $i=1,\dots ,r$ and hence 
$X_f(F)\neq \emptyset$. This completes the
proof of the theorem.\qed
\end{proof}


\section*{Acknowledgments}
  The manuscript was revised during the third named author's stay
  at the IEM, Universit\"at Duisburg-Essen.
  He wishes to thank the IEM for
  kind hospitality and excellent working conditions.
  Yang and Yu were partially supported by grants NSC
  97-2115-M-001-015-MY3, 100-2628-M-001-006-MY4 and AS-99-CDA-M01.
  The authors are grateful to the referees for careful 
  readings and helpful
  comments, specially the suggestion of a referee by bringing into 
  the Hasse principle for homogeneous spaces
  and his/her kind instruction of the proof of Theorem~\ref{74}.

\end{document}